\documentclass[article]{elsarticle}
\usepackage{mathrsfs}
\usepackage{amssymb}
\usepackage{times}
\usepackage{amscd}
\usepackage{amsthm}
\usepackage{amsfonts}
\usepackage{amsmath}
\usepackage{appendix}
\usepackage{relsize}
\usepackage{flafter}
\usepackage{graphicx}
\usepackage{geometry}
\usepackage{mathtools}
\usepackage{hyperref}
\usepackage{cleveref}
\usepackage{lineno,hyperref}
\newcommand*\patchAmsMathEnvironmentForLineno[1]{%
	\expandafter\let\csname old#1\expandafter\endcsname\csname #1\endcsname
	\expandafter\let\csname oldend#1\expandafter\endcsname\csname end#1\endcsname
	\renewenvironment{#1}%
	{\linenomath\csname old#1\endcsname}%
	{\csname oldend#1\endcsname\endlinenomath}}%
\newcommand*\patchBothAmsMathEnvironmentsForLineno[1]{%
	\patchAmsMathEnvironmentForLineno{#1}%
	\patchAmsMathEnvironmentForLineno{#1*}}%
\AtBeginDocument{%
	\patchBothAmsMathEnvironmentsForLineno{equation}%
	\patchBothAmsMathEnvironmentsForLineno{align}%
	\patchBothAmsMathEnvironmentsForLineno{flalign}%
	\patchBothAmsMathEnvironmentsForLineno{alignat}%
	\patchBothAmsMathEnvironmentsForLineno{gather}%
	\patchBothAmsMathEnvironmentsForLineno{multline}%
}

\newtheorem{example}{Example}[section]

\newtheorem{remark}{Remark}[section]

\newtheorem{thm}{Theorem}[section]
\newtheorem{lem}{Lemma}[section]
\newtheorem{cor}{Corollary}[section]
\newtheorem{prop}{Proposition}[section]
\journal{Linear Algebra and its Applications}
\begin{document}
	\begin{frontmatter}
		\title{A log-majorization inequality for normal matrices with applications to determinantal inequalities and geometric means}
		
		\author{Mohammad M. GHABRIES\corref{mycorrespondingauthor}}
		\cortext[mycorrespondingauthor]{Corresponding author}
		\ead[mahdi.ghabries@gmail.com]{mahdi.ghabries@gmail.com}

		\begin{abstract}
			We establish a log-majorization inequality comparing the eigenvalues of the interlaced product $Y^t X^*Y^{1-t}X$ with those of $X^*YX$, valid for every positive semi-definite $Y$ and every normal $X$, with the inequality reversing for $t \notin[0,1]$ when $Y$ is positive definite. This extends known Hermitian results to the strictly larger class of normal matrices, where normality is shown to be the exact structural hypothesis, not a technical convenience. A counterexample proves the result can fail without it. As applications, we settle a normal-matrix extension of a determinantal conjecture of Lin, proving $$\det(A^*A + |BA|^p) \le \det(AA^* + |A^*B^*|^p)$$ for arbitrary $A$, normal $B$ and $p \ge 0$, and we give a complete eigenvalue picture for products of weighted geometric means, sharpening and complementing a theorem of Hiai and Lin.
			\\
		\end{abstract}
		
		\begin{keyword}
			{Log-majorization; Eigenvalues; Convex function; Normal matrix; Hermitian matrix; Positive semi-definite matrix; Determinantal inequalities; Weighted geometric mean}
			\MSC [2020] 15A45, 15A60, 47A64
		\end{keyword}
		
	\end{frontmatter}
	
	\section{Introduction}
	
	Log-majorization has become an essential tool in matrix analysis, operator theory, and their applications. Its power lies in the ability to compare eigenvalues of matrix products, powers, and means in a way that respects the underlying order structure of positive semi-definite matrices. This framework has found widespread use in areas ranging from quantum information theory and diffusion tensor imaging to numerical analysis and control theory.\\
	
	The classical development of log-majorization in the context of matrix means was initiated by the seminal Ando–Hiai theorem \cite{AndoHiai}, which establishes a fundamental monotonicity property stating that the log-majorization inequality
	\[
	(A^p \#_\alpha B^p)^{1/p} \prec_{\log} (A^q \#_\alpha B^q)^{1/q}
	\]
	holds for all positive definite matrices $A,B$, $0 \le \alpha \le 1$, and $0 < q \le p$. This result has since been extended to various matrix means and products; see, for example, \cite{Fur, H} and the references therein. In a related direction, F. Hiai and M. Lin \cite{MHL} proved the log-majorization
	\begin{equation}
		\lambda((A\#_tB)(A\#_{1-t}B)) \prec_{\log} \lambda(AB), \qquad 0 \le t \le 1,
		\label{1}\end{equation}
	which compares the product of two complementary weighted means with the endpoint product $AB$. However, this result was further generalized (see \cite{GAMA2, LS}). Despite these developments, the known results remain confined to the eigenvalue setting for positive definite. In particular, a singular value analogue of \eqref{1} is still open for $t\notin [1/4,3/4]$.\\ 
	
	A parallel and equally active line of investigation concerns determinantal inequalities involving absolute values of matrix products. Audenaert \cite{Audenaert} initiated this direction by proving
	\[
	\det(A^2+|BA|) \leq \det(A^2+AB),\qquad A,B\geq 0,
	\]
	a determinantal inequality motivated by problems in diffusion tensor imaging. M. Lin \cite{Lin} subsequently generalized this to all powers $p\in[0,2]$ and proposed a conjecture comparing $|AB|^p$ with $|BA|^p$ that was eventually settled in the Hermitian setting \cite{GAM,GAMA}. These results, together with their refinements \cite{AG, MAM}, rely heavily on log-majorization techniques applied to Hermitian matrices.\\
	
	The present paper contributes to both of these research directions. Our main result is the log-majorization (Theorem \ref{thm:wlog-main}) that compares the eigenvalues of $Y^t X^* Y^{1-t} X$ with those of $X^* Y X$ when $Y$ is positive semi-definite and $X$ is normal:
	\[
	\lambda(Y^t X^* Y^{1-t} X) \prec_{\log} \lambda(X^* Y X), \qquad 0\le t\le 1.
	\]
	The inequality is reversed for $t\notin[0,1]$ when $Y$ is positive definite. This extends the Hermitian case studied in \cite{MAM, GAM}, where $X$ was required to be self-adjoint, to the strictly larger class of normal matrices. The proof is built on a simple but powerful observation: the convexity of the norm function
	\[
	t \mapsto \|Y^{t/2}X^*Y^{(1-t)/2}\|_\infty
	\]
	on the entire real line, together with the endpoint equality at $t=0$ and $t=1$. Normality of $X$ is not a technical convenience but the exact structural requirement that makes the two endpoints equal; it is the natural hypothesis for the problem.\\
	
	When $X$ is Hermitian, the convexity of this norm function, combined with the symmetry property on $\mathbb{R}$, yields a complete monotonicity picture showing that the function is minimized at $t=1/2$ and increases as $t$ moves away in either direction. This refines the endpoint inequality into a continuous scale of log-majorizations. A counterexample shows that this minimization at $t=1/2$ can fail for general normal matrices in dimensions $n\ge 3$, though it persists in dimension $n=2$.\\
	
	As a concrete illustration of the power of Theorem \ref{thm:wlog-main}, we settle a normal-matrix extension of Lin's conjecture. We prove that for arbitrary $A$ and normal $B$,
	\[
	\det(A^*A + |BA|^p) \leq \det(AA^* + |A^*B^*|^p), \qquad p\geq 0,
	\]
	which reduces to the known Hermitian result \cite{GAM} when $A$ is self-adjoint.\\
	
	As a second application, we specialize the monotonicity results of Section \ref{sec:main} to the setting of weighted geometric means. For positive definite $A,B$, we obtain a complete description of how the eigenvalues of the product $(A\#_tB)(A\#_{1-t}B)$ vary as the weight $t$ moves along the real line (Theorem \ref{thm:geom-monotonicity}). The balanced geometric mean product $(A\#B)^2$ emerges as the unique minimal element in the log-majorization order, while the endpoint product $AB$ serves as the upper bound for $0 \le t \le 1$. More precisely, we establish the following chain refining \eqref{1}
	\[
	\lambda((A\#B)^2) \prec_{\log} \lambda((A\#_tB)(A\#_{1-t}B)) \prec_{\log} \lambda(AB), \qquad 0 \le t \le 1,
	\]
	with the inequalities reversed for $t \notin [0,1]$:
	\[
	\lambda((A\#B)^2) \prec_{\log} \lambda(AB) \prec_{\log} \lambda((A\#_tB)(A\#_{1-t}B)).
	\]
	This complements\eqref{1} by identifying the exact position of the minimum and by providing the monotonicity that interpolates between the two extremes. Moreover, we obtain a further generalization of \eqref{1}: for all $t_1,t_2 \in [0,1]$, we establish
	\[
	\lambda((A\#_{t_1}B)(A\#_{t_2}B)) \prec_{\log} \lambda(A^{2-(t_1+t_2)}B^{t_1+t_2}).
	\]

	Throughout the paper, $\mathbb{M}_n$ denotes the algebra of $n\times n$ complex matrices whose identity matrix is denoted by $I_n$. A matrix $A\in\mathbb{M}_n$ is said to be positive semi-definite (respectively, positive definite), written $A\geq 0$ (respectively, $A>0$), if $x^*Ax\geq 0$ for every $x\in\mathbb{C}^n$ (respectively, $x^*Ax>0$ for every nonzero $x\in\mathbb{C}^n$), and $\mathbb{M}_n^+$ denotes the cone of positive semi-definite matrices. Recall that $B \in \mathbb{M}_n$ is normal if $B^*B = BB^*$; Hermitian matrices (for which $B^* = B$) are a special case, as are skew-Hermitian and unitary matrices. For $A\in\mathbb{M}_n$, the modulus of $A$ is the positive semi-definite matrix $\vert A\vert=(A^*A)^{1/2}$, whose eigenvalues are precisely the singular values of $A$, that is, $\sigma_i(A)=\lambda_i(|A|)$ for all $1\le i\le n$. The L\"owner order on $\mathbb{M}_n$ is defined for Hermitian matrices $A,B$ by $A \ge B$ if and only if $A-B \in \mathbb{M}_n^+$. We refer the reader to \cite{Bhatia, Zhang} for further background. \\
	
	Whenever the eigenvalues $\lambda_1(A),\ldots,\lambda_n(A)$ of a matrix $A$ are real, we shall assume that they are arranged in non-increasing order and we write $\lambda(A)=(\lambda_1(A),\dots,\lambda_n(A))^T$. For any matrices $A, B\in\mathbb{M}_n$ with non-negative eigenvalues, we say that the eigenvalues of $A$ are weakly log-majorized by that of B, denoted by $\lambda(A) \prec_{w\log} \lambda(B)$, if and only if  
	\[
	\prod_{i=1}^k \lambda_i(A)\leq \prod_{i=1}^k \lambda_i(B)\quad \text{for each}  \ \ 1\leq k\leq n.
	\]
	
	We also say $\lambda(A)$ is log-majorized by $\lambda(B)$, denoted by $\lambda(A)\prec_{\log} \lambda(B)$ if the inequalities hold for all $k=1,\ldots,n-1$, with equality for $k=n$.
	\vspace{5mm}
	
	Before presenting our main results, we assemble a few auxiliary lemmas that will serve as the technical backbone of our arguments. The first of these is an elementary convexity fact, stated below for completeness. Despite its simplicity, Lemma \ref{lem:convex-fact} is the engine of everything that follows: once the relevant norm function is shown to be convex on the whole real line with equal values at $0$ and $1$, all the log-majorization relations of this paper, together with their reversals and the monotonicity phenomena, will follow from it.
	
	\begin{lem}\label{lem:convex-fact}
		Let $f: \mathbb{R} \to (0,\infty)$ be convex on $\mathbb{R}$ with $f(0)=f(1)$.
		Then
		\[
		f(t) \le f(0) \le f(s)
		\]
		for all $0 \le t \le 1$ and $s \notin [0,1]$.
	\end{lem}
	
	\begin{proof}
		For $0 \le t \le 1$, convexity gives
		\[
		f(t) \le (1-t)f(0) + t f(1) = f(0).
		\]
		
		For $s > 1$, write $1 = \frac{1}{s} s + \left(1-\frac{1}{s}\right)0$. 
		Convexity on $\mathbb{R}$ yields
		\[
		f(1) \le \frac{1}{s} f(s) + \left(1-\frac{1}{s}\right) f(0).
		\]
		Since $f(0)=f(1)$, rearranging gives $f(s) \ge f(0)$. 
		The case $s < 0$ is analogous.
	\end{proof}
	
	The next lemma connects log-majorization with determinantal inequalities and can be found in \cite[(P2)]{Lin}. It will serve as the bridge from eigenvalues log-majorization relations to the determinantal result of Section \ref{sec:det}.
	
	\begin{lem}\label{lem:log-det}
		Let $X, Y \in M_{n}$ with non-negative eigenvalues. If $\lambda(X)\prec_{w \log} \lambda(Y),$ then 
		\[
		\det(I_n + X) \leq \det(I_n + Y).
		\]
	\end{lem}
	
	The following elementary algebraic identity constitutes one of the building blocks of our main results whose proof can be found in \cite{GAM}. It is a slight generalization of \cite[Lemma A, p. 129]{Tak}.
	
	\begin{lem}\label{lem:algebraic}
		Let $X$ and $Y$ be two invertible matrices. Then, for all $t \in \mathbb{R}$, 
		\[
		(X^*Y^*YX)^{t} = X^*Y^*(YXX^*Y^*)^{t - 1}YX.
		\]
	\end{lem}
	
	For the reader's convenience, we add the final lemma that states the well-known H\"older-type inequality for unitarily invariant norms (see, for example \cite{Bhatia, Kit}): 
	
	\begin{lem}\label{lem:holder}
		Let $A, B \in \mathbb{M}_n^+$ be positive semi-definite and let $X \in \mathbb{M}_n$. Then for any unitarily invariant norm $|||\cdot|||$ and any $\alpha \in [0,1]$,
		\[
		|||A^\alpha X B^{1-\alpha}||| \leq |||A X|||^\alpha \cdot |||X B|||^{1-\alpha}.
		\]
	\end{lem}
	
	With these preliminaries in place, we now turn to our main results.
	
	\section{Main results}\label{sec:main}
	
	We begin with the convexity property that underlies all of our results. The starting point is an extension of the classical convexity property of the norm function $t\mapsto|||A^t X B^{1-t}|||$ from the unit interval to the entire real line. While this function is well known to be convex on $[0,1]$ for positive semi-definite $A,B$, as proved by M. Sababheh \cite{Sab}, the restriction to $[0,1]$ is dictated by the fact that negative powers of singular matrices are undefined. Our first observation is that when $A$ and $B$ are positive definite, nothing is lost by letting $t$ range over the whole real line. We include the short proof, since the extension to the whole real line is precisely what drives our reversed inequalities.
	
	\begin{prop}\label{prop:convexity-R}
		Let $A, B \in \mathbb{M}_n^+$ be positive definite, and let $X \in \mathbb{M}_n$. 
		Then the function
		\[
		f(t) := |||A^t X B^{1-t}|||, \qquad t \in \mathbb{R},
		\]
		is log-convex (and hence convex) on $\mathbb{R}$.
	\end{prop}
	
	\begin{proof}
		If $X = 0$, the result is trivial. Assume $X \neq 0$, so $f(t) > 0$ for all $t$. 
		
		Let $t_1, t_2 \in \mathbb{R}$ be arbitrary, and let $\alpha \in [0,1]$. We need to show
		\[
		f(\alpha t_1 + (1-\alpha)t_2) \le f(t_1)^\alpha f(t_2)^{1-\alpha}.
		\]
		
		The following identity rewrites the argument of the norm in a form convenient to \Cref{lem:holder}:
		\begin{align*}
			f(\alpha t_1 + (1-\alpha)t_2) &= ||| A^{\alpha t_1 + (1-\alpha)t_2} X B^{1-(\alpha t_1 + (1-\alpha)t_2)} |||\\
			&= ||| (A^{t_1-t_2})^{\alpha} (A^{t_2} X B^{1-t_1}) (B^{t_1-t_2})^{1-\alpha} |||.
		\end{align*}
		
		Indeed, the right hand side is a product of three terms, with the middle term independent of $\alpha \in [0,1]$. Applying \Cref{lem:holder} for unitarily invariant norms yields
		\[
		\begin{aligned}
			f(\alpha t_1 + (1-\alpha)t_2)
			&\le ||| A^{t_1-t_2} (A^{t_2} X B^{1-t_1}) |||^{\alpha} \cdot ||| (A^{t_2} X B^{1-t_1}) B^{t_1-t_2} |||^{1-\alpha} \\
			&= ||| A^{t_1} X B^{1-t_1} |||^{\alpha} \cdot ||| A^{t_2} X B^{1-t_2} |||^{1-\alpha} \\
			&= f(t_1)^{\alpha} f(t_2)^{1-\alpha}.
		\end{aligned}
		\]
		
		Since $t_1, t_2 \in \mathbb{R}$ and $\alpha \in [0,1]$ were arbitrary, $f$ is log-convex on $\mathbb{R}$, hence convex on $\mathbb{R}$.
	\end{proof}
	
	The next result is closely related to a number of known inequalities in the literature. The Ando--Hiai--Okubo trace inequality \cite{AHO} and its extension to complex matrices by S. Hayajneh, M. Hayajneh, and F. Kittaneh \cite{HHK} 
	concern products of the alternating form 
	$$M(p,q) = A^p B^q A^{1-p} B^{1-q}.$$ 
	\Cref{thm:wlog-main} is the log-majorization counterpart to these results, now involving the structurally analogous interlaced product \(Y^t X^* Y^{1-t} X\), where $Y$ is positive semi-definite and $X$ is normal. This line of inquiry also parallels the Bourin-type inequalities \cite{Bourin2006} studied by Bhatia \cite{Bhatia2014}, S. Hayajneh and F. Kittaneh \cite{HK2013}, and T. Bottazzi et al. \cite{Bottazzi2015}.
	
	\begin{thm}\label{thm:wlog-main}
		Let $Y\in\mathbb{M}_n$ be a positive semi-definite matrix and let $X\in\mathbb{M}_n$ be a normal matrix. Then, for all $0\le t\le 1$, it holds that
		\begin{equation}\label{eq:wlog-main}
			\lambda(Y^tX^*Y^{1-t}X) \prec_{\log} \lambda(X^*YX).
		\end{equation}
		The inequality is reversed for $t \notin [0,1]$, provided $Y$ is positive definite.
	\end{thm}
	
	\begin{proof}
		First, we define for each $k = 1, 2, \dots, n$
		\[
		X_k := \wedge^k X, \qquad Y_k := \wedge^k Y.
		\]
		Since $Y$ is positive semi-definite, $Y_k$ is also positive semi-definite. Moreover, since $X$ is normal, $X_k$ is also normal:
		\[
		X_k^* X_k = (\wedge^k X)^*(\wedge^k X) = \wedge^k(X^*X) = \wedge^k(XX^*) = (\wedge^k X)(\wedge^k X)^* = X_k X_k^*.
		\]
		
		Using the standard identities for antisymmetric tensor products,
		\[
		Y_k^t X_k^* Y_k^{1-t} X_k = (\wedge^k Y)^t (\wedge^k X)^* (\wedge^k Y)^{1-t} (\wedge^k X) = \wedge^k (Y^t X^* Y^{1-t} X),
		\]
		and
		\[
		X_k^* Y_k X_k = (\wedge^k X)^* (\wedge^k Y) (\wedge^k X) = \wedge^k (X^* Y X).
		\]
		
		Thus we obtain for all $k = 1, 2,\dots, n$
		\[
		\lambda_1(Y_k^t X_k^* Y_k^{1-t} X_k) = \prod_{i=1}^k \lambda_i(Y^t X^* Y^{1-t} X),\ \text{and} \  
		\lambda_1(X_k^* Y_k X_k) = \prod_{i=1}^k \lambda_i(X^*YX).
		\]
		
		Consequently, to establish the desired weak log-majorization it suffices to prove, for each $k = 1, \dots, n$,
		\[
		\lambda_1(Y_k^t X_k^* Y_k^{1-t} X_k) \leq \lambda_1(X_k^* Y_k X_k).
		\]
		
		For $t \in \mathbb{R}$ and $Y > 0$, define
		\[
		A_t := Y^{(1-t)/2} X Y^{t/2}.
		\]
		Then
		\[
		\sigma(A_t)^2 = \lambda(A_t^* A_t)
		= \lambda\big(Y^{t/2} X^* Y^{1-t} X Y^{t/2}\big)
		= \lambda\big(Y^t X^* Y^{1-t} X\big),
		\]
		where the last equality follows from the invariance of eigenvalues under cyclic permutations. In particular, for $t=0$,
		\[
		\sigma(Y^{1/2} X)^2 = \lambda(X^* Y X).
		\]
		
		Applying Proposition \ref{prop:convexity-R} to the spectral norm $\|\cdot\|_{\infty}$ (which is unitarily invariant), the function
		\[
		g(t) = \| Y^{t/2} X^* Y^{(1-t)/2} \|_{\infty}
		\]
		is convex on $\mathbb{R}$. Since $X$ is normal, we get
		\[
		g(0) = \| X^* Y^{1/2} \|_{\infty} = \| Y^{1/2} X \|_{\infty}, \ \text{and} \ \
		g(1) = \| Y^{1/2} X^* \|_{\infty} = \| X Y^{1/2} \|_{\infty} = \| Y^{1/2} X \|_{\infty}.
		\]
		Thus, $g(0)=g(1)=\sigma_1(Y^{1/2}X)$. By Lemma \ref{lem:convex-fact}, for $0 \le t \le 1$,
		\[
		\| Y^{t/2} X^* Y^{(1-t)/2} \|_{\infty} \le \| Y^{1/2} X \|_{\infty}.
		\]
		
		Squaring both sides yields
		\[
		\sigma_1(Y^{t/2} X^* Y^{(1-t)/2})^2 \le \sigma_1(Y^{1/2} X)^2,
		\]
		which is precisely $\lambda_1(Y^t X^* Y^{1-t} X) \le \lambda_1(X^*YX)$. Repeating this argument with $(X_k, Y_k)$ in place of $(X,Y)$ establishes the weak log-majorization
		\[
		\lambda(Y^t X^* Y^{1-t} X) \prec_{w\log} \lambda(X^* Y X), \qquad 0 \le t \le 1.
		\]
		
		To upgrade this to full log-majorization, we observe the determinantal identity
		\[
		\det(Y^t X^* Y^{1-t} X) = \det(X^* Y X).
		\]
		Therefore, the log-majorization inequality \eqref{eq:wlog-main} is obtained.\\
		
		Finally, for $t \notin [0,1]$, assume $Y$ is positive definite. Lemma \ref{lem:convex-fact} gives $g(t) \ge g(0)$ for all $t \notin [0,1]$. Hence
		\[
		\| Y^{t/2} X^* Y^{(1-t)/2} \|_{\infty} \ge \| Y^{1/2} X \|_{\infty}.
		\]
		Following the same procedure as above, we obtain the reversed log-majorization
		\[
		\lambda(Y^t X^* Y^{1-t} X) \succ_{\log} \lambda(X^* Y X)
		\]
		for $t \notin [0,1]$ when $Y$ is positive definite.
	\end{proof}
	
	The first part of the following theorem was proved by L. Plevnik \cite{Plevnik2016}, with positive semi-definite matrices. In addition, the second part extends to normal $X$ the corresponding Hermitian results of \cite{MAM,GAM}.
	
	\begin{thm}\label{thm:extensions}
		Let $Y \in \mathbb{M}_n$ be positive semi-definite and let $X \in \mathbb{M}_n$ be normal. We have for any $p, q \ge 0$,
		\begin{enumerate}
			\item[(1)] $\lambda(Y^p X^* Y^q X) \prec_{\log} \lambda(X^* Y^{p+q} X).$
			\vspace{0.4cm}
			\item[(2)] $\lambda(Y^p X^* Y^{-q} X) \succ_{\log} \lambda(X^* Y^{p-q} X)$, given $Y$ invertible.
		\end{enumerate}
	\end{thm}
	
	\begin{proof}
		For (1), the case $p=q=0$ is trivial, so we assume that $p+q \neq 0$. We simply substitute $Y \mapsto Y^{p+q}$ and take $t = \frac{p}{p+q} \in [0,1]$ in Theorem \ref{thm:wlog-main}, which gives
		\[
		\lambda((Y^{p+q})^{\frac{p}{p+q}} X^* (Y^{p+q})^{1-\frac{p}{p+q}} X) \prec_{\log} \lambda(X^* Y^{p+q} X),
		\]
		which is precisely (1).\\
		
		For the second inequality (2), we introduce three cases:
		
		\begin{enumerate}
			\item[\underline{Case 1:}] Let $p>q$. We replace $Y$ with $Y^{p-q}$ and take $t = \frac{p}{p-q} > 1$ in the reversed inequality of Theorem \ref{thm:wlog-main}. This yields
			\begin{equation}
				\lambda(X^* Y^{p-q} X) \prec_{\log} \lambda(Y^p X^* Y^{-q} X). \label{11}
			\end{equation}
			\item[\underline{Case 2:}] Let $p<q$. It follows directly by substituting $Y$ with $Y^{-1}$ and $X$ with $X^*$, as well as taking $p=q$ and $q=p$ in the first case.
			\item[\underline{Case 3:}] Let $p = q$. Our aim reduces to verifying the majorization
			\begin{equation}
				\lambda(X^*X) \prec_{\log} \lambda(Y^p X^* Y^{-p} X). \label{111}
			\end{equation}
			
			Since $X$ is normal, we have
			\[
			\lambda(X^*X) = |\lambda(X)|^2.
			\]
			
			Now, applying Weyl's majorant inequality (see \cite[Corollary 10.2]{Zhang}), we obtain
			\begin{align*}
				\lambda(Y^p X^* Y^{-p} X) 
				&= \sigma\bigl(Y^{-p/2} X Y^{p/2}\bigr)^2 \\
				&\succ_{\log} \bigl|\lambda\bigl(Y^{-p/2} X Y^{p/2}\bigr)\bigr|^2 \\
				&= |\lambda(X)|^2 \\
				&= \lambda(X^*X),
			\end{align*}
			which establishes~\eqref{111}.
		\end{enumerate}
	\end{proof}

	The endpoint comparisons obtained so far exploit only part of the information contained in the convexity of the function
	\[
	g(t) := \|Y^{t/2} X^* Y^{(1-t)/2}\|_\infty.
	\]
	Convexity in fact implies a richer variational picture. It tells us how the eigenvalues move as the weight slides continuously from one side of the product to the other. The ``valley'' phenomenon of this kind is classical for Heinz-type means: for positive definite $A,B$ and arbitrary $X$, the function
	\[
	\nu \mapsto \|A^\nu X B^{1-\nu} + A^{1-\nu} X B^\nu\|
	\]
	is convex, symmetric about $\nu = 1/2$, decreasing on $[0,1/2]$, and increasing on $[1/2,1]$; the balanced point $\nu = 1/2$ is therefore optimal (see \cite{Kit2010} and the refinements in \cite{AM}). \\
	
	The results below exhibit the same phenomenon for the eigenvalue products of $Y^{t}X^{*}Y^{1-t}X$, thereby providing a continuous family of log-majorizations that interpolates between the two cases of Theorem \ref{thm:wlog-main} and showing that it is only the endpoint of a one-parameter scale of inequalities. However, for the monotonicity results that follow, we require $X$ to be Hermitian rather than merely normal.
	
	\begin{thm}\label{thm:monotonicity}
		Let $Y\in\mathbb{M}_n$ be a positive definite matrix and let $X\in\mathbb{M}_n$ be Hermitian. Then the function
		\[
		g(t):=\|Y^{t/2}XY^{(1-t)/2}\|_\infty,\qquad t\in\mathbb{R},
		\]
		is non-increasing on $(-\infty,1/2]$ and non-decreasing on $[1/2,\infty)$. Consequently, if
		\[
		\left|t_1-\frac{1}{2}\right|
		\le
		\left|t_2-\frac{1}{2}\right|,
		\]
		then
		\[
		\lambda(Y^{t_1}XY^{1-t_1}X)
		\prec_{\log}
		\lambda(Y^{t_2}XY^{1-t_2}X).
		\]
	\end{thm}
	
	\begin{proof}
		Since $X$ is Hermitian, we have
		\[
		g(t) = \|Y^{t/2} X Y^{(1-t)/2}\|_\infty
		= \| (Y^{t/2} X Y^{(1-t)/2})^* \|_\infty
		= \|Y^{(1-t)/2} X Y^{t/2}\|_\infty
		= g(1-t)
		\]
		for all $t \in \mathbb{R}$. Thus $g$ is symmetric about $1/2$. In addition, note that $g$ is convex for any $t \in \mathbb{R}$ by Proposition \ref{prop:convexity-R}. So, convexity and symmetry of $g$ give
		\[
		g(t) = \frac{g(t) + g(1-t)}{2} \ge g\left(\frac{t + (1-t)}{2}\right) = g\big(\frac{1}{2}\big).
		\]
		Hence $g$ attains its global minimum at $t = 1/2$. A convex function with a minimum at $m$ is non-increasing on $(-\infty, m]$ and non-decreasing on $[m, \infty)$; applying this with $m = 1/2$ yields the claimed monotonicity.\\
		
		Since $g(t)=g(1-t)$, we have $g(t)=g\big(\frac{1}{2}+|t-\frac{1}{2}|\big)$ for all $t\in\mathbb{R}$. \\
		
		Next, we define $$h(s)=g\left(\frac{1}{2}+s\right), \ \ \ \ \ \ \text{for} \ s\ge 0.$$ 
		
		As $g$ is increasing on $[\frac{1}{2},\infty)$, 
		$h$ is increasing on $[0,\infty)$. Hence $g(t)=h\big(|t-\frac{1}{2}|\big)$, and therefore
		\[
		\Big|t_1-\frac{1}{2}\Big|\le\Big|t_2-\frac{1}{2}\Big|\;\Longrightarrow\; g(t_1)\le g(t_2).
		\]
		
		To obtain the log-majorization, square the inequality $g(t_1) \le g(t_2)$ and apply \Cref{thm:wlog-main} to the exterior powers $X_k = \wedge^k X$ and $Y_k = \wedge^k Y$ for each $k = 1, \dots, n$. This implies
		\[
		\lambda(Y^{t_1} X Y^{1-t_1} X) \prec_{\log} \lambda(Y^{t_2} X Y^{1-t_2} X),
		\]
		completing the proof.
	\end{proof}
	
	In dimension $n=2$, the conclusion of \Cref{thm:monotonicity} holds for \emph{every} normal $X$. 
	
	\begin{prop}\label{prop:normal-2d}
		Let $Y\in\mathbb{M}_2$ be positive definite and let $X\in\mathbb{M}_2$ be normal. Then the function
		\[
		g(t) = \|Y^{t/2} X Y^{(1-t)/2}\|_\infty, \qquad t\in\mathbb{R},
		\]
		satisfies $g(t) = g(1-t)$ and is convex. Consequently, $g$ is decreasing on $(-\infty,1/2]$ and increasing on $[1/2,\infty)$, and for any $t_1,t_2\in\mathbb{R}$,
		\[
		\Big|t_1 - \frac{1}{2}\Big| \le \Big|t_2 - \frac{1}{2}\Big|
		\;\Longrightarrow\;
		\lambda(Y^{t_1} X Y^{1-t_1} X) \prec_{\log} \lambda(Y^{t_2} X Y^{1-t_2} X).
		\]
	\end{prop}
	
	\begin{proof}
		For the sake of simplicity, we define 
		\[
		M(t) := Y^t X^* Y^{1-t} X.
		\]
		Since $Y$ is positive definite, there exists a unitary $U$ such that
		\[
		U^* Y U = \begin{pmatrix} y_1 & 0 \\ 0 & y_2 \end{pmatrix}, \qquad y_1, y_2 > 0.
		\]
		Replacing $X$ by $U^* X U$ does not affect the eigenvalues (nor the normality), so we may assume without loss of generality that
		\[
		Y = \begin{pmatrix} y_1 & 0 \\ 0 & y_2 \end{pmatrix}, \qquad y_1, y_2 > 0.
		\]
		
		For any $X = \begin{pmatrix} a & b \\ c & d \end{pmatrix}$ normal, it is well known that $|b| = |c|$. A direct computation yields
		\[
		\operatorname{tr} M(t) = |a|^2 y_1 + |d|^2 y_2 + |b|^2\left(y_1^t y_2^{1-t} + y_2^t y_1^{1-t}\right),
		\]
		which is symmetric about $t=1/2$, that is, $\operatorname{tr} M(t) = \operatorname{tr} M(1-t)$. Moreover,
		\[
		\det M(t) = y_1 y_2 |\det X|^2.
		\]
		
		For a matrix $2\times 2$ with nonnegative eigenvalues and constant determinant, and $\operatorname{tr} M(t) = \operatorname{tr} M(1-t)$, we conclude that $M(t)$ and $M(1-t)$ have the same characteristic polynomial. Hence
		\[
		\lambda_i(M(t)) = \lambda_i(M(1-t)), \ \ \ \ \ \ \ \ i=1,2
		\] and consequently $g(t) = g(1-t)$. The symmetry argument used in the proof of Theorem \ref{thm:monotonicity} then applies, establishing that $g$ is decreasing on $(-\infty,1/2]$ and increasing on $[1/2,\infty)$. The log-majorizations follow exactly as in Theorem \ref{thm:monotonicity}.
	\end{proof}
	
	However, for $n \geq 3$, these properties can fail for non-Hermitian normal matrices, as the following example demonstrates.
	
	\begin{example}\label{ex:n=3-counterexample}
		Let
		\[
		Y = \begin{pmatrix}
			3 & 0 & 0\\
			0 & 2 & 0\\
			0 & 0 & 1
		\end{pmatrix},
		\qquad
		X = \begin{pmatrix}
			0 & 1 & 0\\
			0 & 0 & 1\\
			1 & 0 & 0
		\end{pmatrix}.
		\]
		The matrix \(X\) is a permutation matrix, hence unitary and therefore normal. A direct numerical computation gives
		\[
		g(0.3)=\|Y^{0.15}XY^{0.35}\|_\infty\approx1.5029,
		\]
		while
		\[
		g(0.7)=\|Y^{0.35}XY^{0.15}\|_\infty\approx1.6298.
		\]
		Hence,
		\[
		g(0.3)\neq g(0.7),
		\]
		showing that the symmetry $g(t) = g(1-t)$ fails for normal \(X\) in general.\\
		
		Furthermore,
		\[
		g\!\left(\frac{1}{2}\right)
		=\|Y^{1/4}XY^{1/4}\|_\infty
		\approx1.5651
		>1.5029=g(0.3),
		\]
		so $t=\frac{1}{2}$ is not a minimizer of $g$. In fact, a numerical search indicates that the minimum is attained near $t\approx0.27$, where
		\[
		g(0.27)\approx1.4938.
		\]
	\end{example}
	
	Nevertheless, Theorem \ref{thm:wlog-main} still guarantees that $g(t) \leq g(0) = g(1)$ for all $t \in [0,1]$, so the endpoint $t=0$ (equivalently $t=1$) remains maximal. However, the minimum is no longer guaranteed to occur at the balanced point $t = 1/2$.\\
	
	The following corollary translates the monotonicity of $g(t) := \|Y^{t/2} X Y^{(1-t)/2}\|_\infty$ into a statement about products of the form $Y^p X Y^q X$ with arbitrary non-negative exponents $p, q$. For $p, q \in \mathbb{R}$, define
	\[
	F(p,q) := \lambda(Y^p X Y^q X).
	\]
	
	\begin{cor}\label{cor:pq-monotonicity}
		Let $Y \in \mathbb{M}_n^+$ and $X \in \mathbb{M}_n$ be Hermitian. Let $p_1, p_2, q_1, q_2\ge0$ such that $p_1+q_1= p_2+q_2 =s$. If
		\[
		\left|p_1-\frac{s}{2}\right|
		\le
		\left|p_2-\frac{s}{2}\right|,
		\]
		then
		\[
		F(p_1,q_1)\prec_{\log}F(p_2,q_2).
		\]
	\end{cor}
	
	\begin{proof}
		The case $s=0$ is trivial, so we assume $s>0$. For $i=1,2$, we set $t_i=\frac{p_i}{s}\in[0,1].$ Since $p_i+q_i=s$, we have
		$q_i=s(1-t_i)$, and therefore
		\[
		F(p_i,q_i)
		=\lambda\!\left((Y^s)^{t_i}X(Y^s)^{1-t_i}X\right).
		\]
		
		Now, substituting $Y$ with $Y^s$ in \Cref{thm:monotonicity} yields the required log-majorization inequality for
		\[
		\left|p_1-\frac{s}{2}\right|
		\le
		\left|p_2-\frac{s}{2}\right|, 
		\]
		as required.
	\end{proof}

	The next result extends the monotonicity theorem to products involving negative exponents. This corresponds to the case $t \notin [0,1]$ in \Cref{thm:monotonicity}, and it may be viewed as a quantitative sharpening of Theorem \ref{thm:extensions}(2): not only does the negative-exponent product dominate $\lambda(X^* Y^{p-q} X)$ in the log-majorization order, but this domination grows monotonically as the exponents move away from the balanced configuration. 
	
	\begin{cor}\label{cor:negative-monotonicity}
		Let $Y$ be positive definite and $X$ be Hermitian. For $p_1,q_1,p_2,q_2>0$ with $r = p_1-q_1 = p_2-q_2 \neq 0$, if $p_1 \le p_2$, then
		\begin{equation}
			F(p_1,-q_1) \prec_{\log} F(p_2,-q_2). \label{31}
		\end{equation}
	\end{cor}
	
	\begin{proof}
		Set $t_i = p_i/r$, so $1-t_i = -q_i/r$. Then
		\[
		F(p_i,-q_i) = \lambda\big( (Y^r)^{t_i} X (Y^r)^{1-t_i} X \big).
		\]
		If $r>0$, then $t_i > 1$, and $p_1 \le p_2$ gives $t_1 \le t_2$. Since $t_i \ge 1/2$, \Cref{thm:monotonicity} applied to $Y^r$ yields the result. Similarly, if $r<0$, then $t_i < 0$ and $p_1 \le p_2$ implies $t_1 \ge t_2$. Since $t_i \le 1/2$, \Cref{thm:monotonicity} applied to $Y^r$ again, establishing \eqref{31}.
	\end{proof}
	
	\section{Determinantal inequalities for normal matrices}\label{sec:det}
	
	We now turn to a concrete application of the log-majorization machinery developed in the preceding section. The starting point is a line of determinantal inequalities initiated by Audenaert \cite{Audenaert}, who proved that for $A,B\geq 0$,
	\begin{equation}\label{eq:audenaert}
		\det(A^2+|BA|) \leq \det(A^2+AB),
	\end{equation}
	motivated by problems in diffusion tensor imaging. M. Lin \cite{Lin} generalized this to
	\begin{equation}\label{eq:lin}
		\det(A^2+|BA|^p) \leq \det(A^2+A^pB^p),\qquad 0\leq p\leq 2,
	\end{equation}
	and proposed several open problems including
	\begin{equation}\label{conj}
		\det(A^2+|AB|^p) \geq \det(A^2+|BA|^p),\qquad A,B\geq 0,\; p\geq 0.
	\end{equation}
	
	The conjectures generated considerable activity and were eventually confirmed in wider settings \cite{GAM,GAMA}. Further determinantal inequalities of this type can be found in \cite{AG,MAM}.\\
	
	In view of these developments, it is natural to ask how far beyond the Hermitian setting such inequalities remain valid. Using Theorem \ref{thm:extensions}, we show that \eqref{conj} extends to the case where $A$ is an \emph{arbitrary} complex matrix and $B$ is a \emph{normal} matrix, provided the two sides are appropriately symmetrized.
	
	\begin{thm}\label{thm:det-adjoint}
		Let $A\in \mathbb{M}_n$ be arbitrary and let $B\in\mathbb{M}_n$ be a normal matrix. Then for all $p\geq 0$, it holds that
		\[
		\det(A^*A + |BA|^p) \leq \det(AA^* + |A^*B^*|^p).
		\]
	\end{thm}
	
	\begin{proof}
		First, assume $A$ is invertible. Then for all $p \geq 0$,
		\begin{align*}
			\lambda\left( (A^*)^{-1} (A^*B^*BA)^{\frac{p}{2}} A^{-1}\right) 
			&= \lambda\left(B^* (BAA^*B^*)^{\frac{p}{2} - 1} B\right) \quad \text{(by \Cref{lem:algebraic})}\\
			&\prec_{\log} \lambda\left((BAA^*B^*)^{\frac{p}{2}} B^* (BAA^*B^*)^{-1} B\right) \quad \text{(by \Cref{thm:extensions}(2))}\\
			&= \lambda\left((BAA^*B^*)^{\frac{p}{2}} (A^*)^{-1}A^{-1}\right)\\
			&= \lambda\left(A^{-1} (BAA^*B^*)^{\frac{p}{2}} (A^*)^{-1}\right).
		\end{align*}
		
		Applying Lemma \ref{lem:log-det} gives
		\[
		\det\left(I_n + (A^*)^{-1} (A^*B^*BA)^{\frac{p}{2}} A^{-1}\right) \leq \det\left(I_n + A^{-1} (BAA^*B^*)^{\frac{p}{2}} (A^*)^{-1}\right).
		\]
		
		Multiplying both sides by $\det(A^*A) > 0$ yields
		\[
		\det(A^*A + |BA|^p) \leq \det(AA^* + |A^*B^*|^p).
		\]
		
		Finally, for $A$ singular, a standard continuity argument applies. There exists $\delta > 0$ such that for all $0 < \epsilon < \delta$, the perturbation $$A_\epsilon = A + \epsilon I_n$$ is invertible. Since the inequality
		\[
		\det(A_\epsilon^* A_\epsilon + |B A_\epsilon|^p) \leq \det(A_\epsilon A_\epsilon^* + |A_\epsilon^* B^*|^p)
		\]
		holds for each $A_\epsilon > 0$, and the functions $(A,B) \mapsto |B A|^p$ and $(A,B)\mapsto |A^* B^*|^p$  for $p \geq 0$ together with the determinant are continuous on $\mathbb{M}_n$, letting $\epsilon \to 0^+$ completes the proof.
	\end{proof}
	
	Specializing \Cref{thm:det-adjoint} to Hermitian $A$ and $B$ recovers the main determinantal result of \cite{GAM}. 
	
	\begin{remark}
		It is worth emphasizing the extent of the generalization. In \cite{GAM}, both $A$ and $B$ were required to be Hermitian. Here, $B$ may range over the entire class of normal matrices, and $A$ is completely arbitrary. The normality assumption on $B$ cannot simply be dropped, as \Cref{ex:normal-needed} confirms.
	\end{remark}
	
	\begin{example}\label{ex:normal-needed}
		Take $p = 2$, and
		\[
		A = \begin{pmatrix} 2 & 0 \\ 0 & 1 \end{pmatrix}, \qquad
		B = \begin{pmatrix} 1 & 1 \\ 0 & 1 \end{pmatrix}.
		\]
		Here $B$ is not normal, and one computes
		\[
		\det(A^*A + |BA|^2) = 20, \qquad \det(AA^* + |A^*B^*|^2) = 17,
		\]
		so the inequality of \Cref{thm:det-adjoint} fails. 
	\end{example}

	\section{Monotonicity for products of weighted geometric means}\label{sec:geom-mean}
	
	In this section, we apply the monotonicity results of Section \ref{sec:main} to derive a new log-majorization inequality for products of weighted geometric means. This provides a second, independent application of our main theorem and demonstrates that its reach extends well beyond determinantal inequalities, into the theory of matrix means.\\
	
	The matrix geometric mean has a rich history. It was first introduced by W. Pusz and S.L. Woronowicz \cite{PW} in 1975, in the general setting of sesquilinear forms, and was subsequently studied in depth by T. Ando \cite{Ando79}, who established its fundamental order-theoretic and variational properties. For positive definite matrices $A,B\in\mathbb{M}_n$, the geometric mean is given explicitly by
	\[
	A\#B = A^{1/2}\big(A^{-1/2}BA^{-1/2}\big)^{1/2}A^{1/2},
	\]
	and it is the unique positive definite solution of the Riccati equation $ZA^{-1}Z = B$. The theory of operator means was developed by F. Kubo and T. Ando \cite{KA}.\\
	
	More generally, for positive definite $A,B$, the \emph{weighted geometric mean} is defined by
	\begin{equation}\label{eq:wgm-def}
		A\#_t B = A^{1/2}\big(A^{-1/2}BA^{-1/2}\big)^{t}A^{1/2} \qquad \text{for  } t\in\mathbb{R},
	\end{equation}
	so that $A\#_0B = A$, $A\#_1B = B$ and $A\#_{1/2}B = A\#B$. The weighted mean carries a beautiful geometric interpretation: the set of positive definite matrices forms a Riemannian manifold of non-positive curvature under the affine-invariant metric, and $t\mapsto A\#_tB$ is precisely the unique geodesic joining $A$ to $B$, with $A\#B$ as its midpoint (see \cite{Bhatia2007,LL}). This geodesic viewpoint has been a driving force in the modern theory of matrix means and their multivariable extensions. \\
	
	Some care is needed regarding the domain of definition. When $A$ and $B$ are merely positive semi-definite, the expression \eqref{eq:wgm-def} is no longer available, since negative powers of singular matrices are undefined. Nevertheless, for the weight range $0\le t\le 1$ the weighted geometric mean extends naturally to all $A,B\in\mathbb{M}_n^+$ via the strong limit
	\[
	A\#_t B := \lim_{\varepsilon\rightarrow 0}\, (A+\varepsilon I_n)\,\#_t\,(B+\varepsilon I_n), \qquad 0\le t\le 1,
	\]
	which exists as a decreasing limit by the monotonicity of operator means \cite{KA}; this is the standard Kubo--Ando construction. For $t\notin[0,1]$, however, the expression \eqref{eq:wgm-def} genuinely requires $A$ and $B$ to be positive definite, and we shall keep this distinction in view throughout.\\
	
	More recently, eigenvalue comparisons for products of geometric means have attracted considerable interest. F. Hiai and M. Lin \cite{MHL} proved the elegant log-majorization
	\begin{equation}
		\lambda((A\#_tB)(A\#_{1-t}B)) \prec_{\log} \lambda(AB), \qquad 0\leq t\leq 1.\label{41}
	\end{equation}
	
	Subsequent work by R. Lemos and G. Soares \cite{LS} established further log-majorization inequalities for matrix connections, while related results for eigenvalues and singular values were obtained in \cite{GAMA2}. The question of how the eigenvalues in \eqref{41} vary as $t$ moves inside the interval had remained open; the following theorem settles it.

	\begin{thm}\label{thm:geom-monotonicity}
		Let $A,B$ be positive definite matrices. Then, for all $t_1, t_2 \in \mathbb{R}$ with
		$\left|t_1-\frac{1}{2}\right|
		\le
		\left|t_2-\frac{1}{2}\right|,$
		\begin{equation}
			\lambda((A\#_{t_1}B)(A\#_{1-t_1}B)) \prec_{\log} \lambda((A\#_{t_2}B)(A\#_{1-t_2}B)).\label{42}
		\end{equation}
	\end{thm}
	
	\begin{proof}
		The key observation is that the eigenvalues of the product $(A\#_tB)(A\#_{1-t}B)$ can be written in the form $Y^t X Y^{1-t} X$ with $X$ positive definite. Indeed, setting $Y = A^{-1/2} B A^{-1/2}$, we have
		\[
		A\#_tB = A^{1/2} Y^t A^{1/2},
		\]
		and hence
		\[
		\lambda((A\#_tB)(A\#_{1-t}B)) = \lambda(A^{1/2} Y^t A^{1/2} \cdot A^{1/2} Y^{1-t} A^{1/2}) = \lambda(Y^t A Y^{1-t} A).
		\]
		Applying \Cref{thm:monotonicity} with $X = A$ and $Y = A^{-1/2}BA^{-1/2}$, gives precisely \eqref{42}. 
	\end{proof}
	
	\vspace{0.4cm}
	
	The next corollary shows that the balanced geometric mean product $(A\#B)^2$ is the unique minimal element in the log-majorization order among all products $(A\#_tB)(A\#_{1-t}B)$, while $AB$ serves as the upper endpoint for $0\le t\le 1$.
	
	\begin{cor}
		Let $A,B$ be positive definite matrices. Then:
		\begin{enumerate}
			\item[(1)] For all $0\le t \le 1$,
			\[
			\lambda((A\#B)^{2}) \prec_{\log}  \lambda((A\#_{t}B)(A\#_{1-t}B)) \prec_{\log} \lambda(AB).
			\]
			
			\item[(2)] For all $t \notin [0,1]$,
			\[
			\lambda((A\#B)^{2}) \prec_{\log} \lambda(AB) \prec_{\log} \lambda((A\#_{t}B)(A\#_{1-t}B)).
			\]
		\end{enumerate}
	\end{cor}
	
	\begin{proof}
		First we take $t_1 = t$ and $t_2 = 1/2$ in \Cref{thm:geom-monotonicity}. This implies
		\[
		\lambda((A\#B)^2) \prec_{\log} \lambda((A\#_tB)(A\#_{1-t}B)), \qquad t\in\mathbb{R}.
		\]
		Combining this with the Hiai--Lin theorem \eqref{41} for $t\in[0,1]$ gives 
		\[
		\lambda((A\#B)^2) \prec_{\log} \lambda((A\#_tB)(A\#_{1-t}B)) \prec_{\log} \lambda(AB).
		\]
		
		For (2), we substitute $t_1 = 1$ and $t_2 = t$ in \Cref{thm:geom-monotonicity}. This yields the full chain, that is, for $t\notin [0,1]$, we have
		\[
		\lambda((A\#B)^2) \prec_{\log} \lambda(AB) \prec_{\log} \lambda((A\#_tB)(A\#_{1-t}B)).
		\]
	\end{proof}

	Geometrically, the further the pair of weights $(t,1-t)$ drifts from the midpoint of the geodesic joining $A$ to $B$, the larger the eigenvalue products become in the log-majorization order. This is a matrix-mean analogue of the classical valley phenomenon for Heinz means recalled in Section \ref{sec:main}. In particular, this extends the trace inequality obtained by Bhatia, Lim and Yamazaki \cite{BLY},
	\[
	\operatorname{tr}((A\#_tB)(A\#_{1-t}B)) \leq \operatorname{tr}(AB) \ \ \ \ \ \text{for every} \ \ 0\leq t \leq 1. 
	\]
	
	The inequality is thereby upgraded from a trace inequality to a log-majorization, sharpened by the identification of $t = 1/2$ as the exact minimizer, and reversed for $t\notin [0,1]$.\\
	
	Our next result, which presents the final main contribution of this section, provides a complete comparison between the product of two arbitrary weighted geometric means and a suitable power product of the endpoint matrices. This theorem includes the Hiai--Lin inequality as a special case and offers a unified framework for understanding how the eigenvalues of $(A\#_{t_1}B)(A\#_{t_2}B)$ behave across the entire parameter square $[0,1]\times[0,1]$.\\
	
	The following lemma, due to \cite{AG}, is a key ingredient in our proof.
	
	\begin{lem} \label{Lemma 4.1}
		Let $A$, $B$ be two positive semi-definite matrices with $A$ invertible. Then for all $0 \le t\le 1$ and for all $k \ge 1$, $$\lambda(A^{kt/2}(A^{-1/2}BA^{-1/2})^t A^{kt/2}) \prec_{\log} \lambda(A^{(k-1)t} B^t).$$
	\end{lem}
	
	For the sake of simplicity, we shall use the notation
	\[
	C := A^{-\frac{1}{2}} B A^{-\frac{1}{2}},
	\]
	whenever $A$ and $B$ are fixed positive definite matrices in $\mathbb{M}_n$.	
	
	\begin{thm}\label{Theorem 4.2}
		Let $A$ and $B$ be two $n\times n$ positive semi-definite matrices. Then $$\boldsymbol{\lambda}\big{(} (A\#_{t_1}B)(A\#_{t_2}B) \big{)} \prec_{log} \boldsymbol{\lambda}\big( A^{2 - (t_1 + t_2)} B^{t_1 + t_2}\big), \hspace{1cm} 0\leq t_1, t_2\leq 1.$$
	\end{thm}
	
	\begin{proof}  
		We assume that $A$ is a positive definite matrix, and the general case can then be deduced by a continuity argument. In addition, we divide the proof into two cases. \begin{enumerate}
			\item[\underline{Case 1:}] If $0\leq t_1 + t_2 \leq 1$, then a direct computation shows \begin{align*} \boldsymbol{\lambda}\big( (A\#_{t_1}B)(A\#_{t_2}B) \big) &= \boldsymbol{\lambda}\left( A^{\frac{1}{2}} (A^{-\frac{1}{2}} B A^{-\frac{1}{2}})^{t_1} A (A^{-\frac{1}{2}} B A^{-\frac{1}{2}})^{t_2} A^{\frac{1}{2}} \right)\\
				&= \boldsymbol{\lambda}\left( A (A^{-\frac{1}{2}} B A^{-\frac{1}{2}})^{t_1} A (A^{-\frac{1}{2}} B A^{-\frac{1}{2}})^{t_2} \right)\\
				&\prec_{log} \boldsymbol{\lambda}\left( A^2 (A^{-\frac{1}{2}} B A^{-\frac{1}{2}})^{t_1 + t_2} \right)\\
				&= \boldsymbol{\lambda}\left( (A^{\frac{1}{2}})^2 (A^{-\frac{1}{2}} B A^{-\frac{1}{2}})^{t_1 + t_2} (A^{\frac{1}{2}})^2\right)\\
				&\prec_{log} \boldsymbol{\lambda}\left( A^{2 - (t_1 + t_2)} B^{t_1 + t_2} \right).
			\end{align*}
			
			Here, the first inequality follows from \Cref{thm:extensions}(1) by replacing $X$ and $Y$ with $A$ and $A^{-\frac{1}{2}} B A^{-\frac{1}{2}}$, respectively, and the last inequality follows from taking $0\leq t = t_1 + t_2 \leq 1$ and $k = \frac{2}{t_1 + t_2} > 1$ in \Cref{Lemma 4.1}.\vspace{0.3cm}
			
			\item[\underline{Case 2:}] Let $0\leq t_1, t_2\leq 1$ such that $s= t_1 + t_2 \geq 1$. One checks that \[\lambda_1\left( A^{2 - s} B^{s}\right) = \lambda_1\left( A^{1 - \frac{s}{2}} B^{s} A^{1 - \frac{s}{2}}\right) \ \ \text{and} \ \ \lambda_1\left[ (A\#_{t_1}B)(A\#_{t_2}B) \right] = \lambda_1( C^{\frac{t_2}{2}} A C^{t_1} A C^{\frac{t_2}{2}}).\]
			
			Both sides are homogeneous of degree $s$ in $B$, so we may normalize $\lambda_1\left( A^{1 - \frac{s}{2}} B^{s} A^{1 - \frac{s}{2}}\right) \leq 1$; it then suffices to show the implication $$A^{1 - \frac{s}{2}} B^{s} A^{1 - \frac{s}{2}} \leq I_n \Rightarrow C^{\frac{t_2}{2}} A C^{t_1} A C^{\frac{t_2}{2}} \leq I_n.$$
			
			Assume then that $A^{1 - \frac{s}{2}} B^{s} A^{1 - \frac{s}{2}} \leq I_n$. This implies that $$B^{s} \leq A^{s - 2}.$$
			
			By L\"owner--Heinz inequality for $0\leq \frac{1}{s}\leq 1$, we obtain \begin{equation} B \leq A^{1 - \frac{2}{s}}.\label{4.4.1}\end{equation}
			
			Next, conjugating both sides of \eqref{4.4.1} with $A^{-\frac{1}{2}} > 0$ yields $$C \leq A^{-\frac{2}{s}},$$ and by appealing again to L\"owner--Heinz inequality this time  for $0\leq t_1\leq 1$ and $0\leq t_2\leq 1$, respectively, we obtain \begin{equation} C^{t_1} \leq A^{-\frac{2t_1}{s}},\label{4.4.2}\end{equation} and \begin{equation}C^{t_2} \leq A^{-\frac{2t_2}{s}}.\label{4.4.3}
			\end{equation}
			
			Now, observe that  \begin{align*} C^{t_2/2} A C^{t_1} A C^{t_2/2} &\leq C^{t_2/2} A A^{-\frac{2t_1}{s}} A C^{t_2/2} \ \ \ \ \ \ \ \ \ \text{(by using \eqref{4.4.2})}\\
				&= C^{t_2/2} A^{\frac{2t_2}{s}} C^{t_2/2}\\
				&\leq I_n \hspace{3cm} \ \ \ \ \ \ \ \ \ \text{(by using \eqref{4.4.3})}. \end{align*}
			
			Therefore, for all $0\leq t_1, t_2\leq 1$ such that $s \geq 1$, \begin{equation}\lambda_1\left[ (A\#_{t_1}B)(A\#_{t_2}B) \right] \leq \lambda_1\left( A^{2 - s} B^{s}\right).\label{4.4.4}\end{equation}
			
			Using the anti-symmetric tensor product, we have for $1\leq r\leq n$ $$\wedge^r\left( A^{2 - s} B^{s}\right) = \left(\wedge^r A\right)^{2 - s} \left(\wedge^r B\right)^{s},$$ and $$\wedge^r\left[ (A\#_{t_1}B)(A\#_{t_2}B) \right] = \left(\wedge^r A \#_{t_1} \wedge^r B\right) \left(\wedge^r A \#_{t_2} \wedge^r B\right).$$
			
			Replacing $A$ and $B$ with $\wedge^r A$ and $\wedge^r B$, respectively, in \eqref{4.4.4} gives $$\lambda_1\left[ \left(\wedge^r A \#_{t_1} \wedge^r B\right) \left(\wedge^r A \#_{t_2} \wedge^r B\right)\right] \leq \lambda_1\left[ \left(\wedge^r A\right)^{2 - s} \left(\wedge^r B\right)^{s}\right].$$
			
			This is equivalent to $$\lambda_1\left[ \wedge^r\big( (A\#_{t_1}B)(A\#_{t_2}B) \big)\right] \leq \lambda_1\left[ \wedge^r\left( A^{2 - s} B^{s}\right)\right].$$
			
			So, for all $1\leq r\leq n$, we have $$\prod\limits_{i=1}\limits^r \lambda_i\left[ (A\#_{t_1}B)(A\#_{t_2}B) \right] \leq \prod\limits_{i=1}\limits^r \lambda_i( A^{2 - s} B^{s}).$$
			
			To conclude, observe that $\det\left[ (A\#_{t_1}B)(A\#_{t_2}B) \right] = \det( A^{2 - s} B^{s})$.\\
			
			This completes Case 2, and with it the proof.
		\end{enumerate}
	\end{proof}

	\section*{Declaration of Competing Interest}
	The author declares that there is no competing interest.

\end{document}